\documentclass[12pt]{article}

\usepackage[english]{babel}
\usepackage{upref,amsfonts,amsxtra,latexsym,color}
\usepackage{amssymb,amsmath,amsthm,a4wide,epsf,mathrsfs,verbatim,hyperref}
\usepackage[all]{xy}

\newtheorem{theorem}{Theorem}[subsection]
\newtheorem{lemma}[theorem]{Lemma}
\newtheorem{corollary}[theorem]{Corollary}
\newtheorem{proposition}[theorem]{Proposition}

\theoremstyle{definition}
\newtheorem{definition}[theorem]{Definition}

\newtheorem{example}[theorem]{Example}
\newtheorem{question}[theorem]{Question}

\numberwithin{equation}{section}
\numberwithin{theorem}{section}

\newcommand{\diag}{\mathrm{diag}}

\newcommand{\Ad}{\mathrm{Ad}\,}

\newcommand{\cM}{{\cal M}}
\newcommand{\ep}{{\varepsilon}}
 
\newcommand{\cA}{{A}}

\newcommand{\cU}{{\cal U}}

\newcommand{\cB}{{B}}

\newcommand{\cO}{{\cal O}}
\newcommand{\cD}{{D}}

\newcommand{\C}{{\mathbb C}}
\newcommand{\Z}{{\mathbb Z}}

\newcommand{\T}{\mathbb{T}}
\newcommand{\cK}{{\cal K}}

\newcommand{\Cs}{{$C^*$-al\-ge\-bra}}

\newcommand{\sh}{{$^*$-ho\-mo\-mor\-phism}}

\newcommand{\AvP}{{\rm{(AveP)}}}
\newcommand{\SAvP}{{\rm{(SAveP)}}}
\newcommand{\fin}{\mathrm{fin}}
\newcommand{\Inf}{\mathrm{inf}}

\date{}

\title{A Dixmier type averaging property of automorphisms on a $C^*$-algebra}

\author{Mikael R\o rdam} 

\begin{document}

\maketitle

\begin{abstract} 
\noindent In his study of the relative Dixmier property for inclusions of von Neumann algebras and of \Cs s, Popa considered a certain property of automorphisms on \Cs s, that we here call the \emph{strong averaging property}.  In this note we  characterize when an automorphism on a \Cs{} has the strong averaging property.  In particular,  automorphisms on commutative \Cs s possess this property precisely when they are free. An automorphism on a unital separable simple \Cs{} with at least one tracial state has the strong averaging property precisely when  its extension to the finite part of the bi-dual of the \Cs{} is properly outer, and in the simple, non-tracial case the strong averaging property is equivalent to being outer. 

To illustrate the usefulness of the strong averaging property we give three examples where we can provide  simpler proofs of existing results on crossed product \Cs s, and we are also able to extend these results in different directions.  
\end{abstract}

\section{Introduction} 

\noindent There are several notions of ``non-innerness'' of an automorphism on a \Cs{} (or a von Neumann algebra), most notably outerness, proper outerness and freeness (which on commutative \Cs s is equivalent to proper outerness). These properties of an automorphism or a group action, in turn, facilitate properties of the associated dynamical system and of the crossed product, such as for example simplicity (for \Cs s) and factoriality (for von Neumann algebras). 

We shall  investigate a property of automorphisms considered by Popa in his paper~\cite{Popa:JFA2000} where he studies the relative Dixmier property for an inclusion of \Cs s (or von Neumann algebras), see Definition~\ref{def:AvP}. We call this property the (strong) averaging property of the automorphism as it resembles the Diximier averaging property. Popa proves, via establishing this property, that the inclusion of a simple \Cs{} with the Dixmier property inside its crossed product by an arbitrary discrete group equipped with a suitable outer action on the \Cs{} has the relative Dixmier property, \cite[Corollary 4.1]{Popa:JFA2000}.

The goal of this paper is to characterize which automorphisms enjoy the (strong) averaging property, and to demonstrate its usefulness via some examples, in particular in the direction of simplifying existing proofs.

The characterization of the strong averaging property, here abbreviated \SAvP, is given in terms of residual outerness of the automorphism (where ``residual'' means that the outerness property must hold in all invariant quotients), and in terms of suitable outerness of the extension of the automorphism to the bi-dual of the \Cs. More specifically, we must require that the automorphism is properly outer on the finite part of the bi-dual. We remind the reader that proper outerness of an automorphism on a \Cs{} does not imply that its extension to the bi-dual also is properly outer. However, by a result of Kishimoto, \cite{Kis:Auto}, if the \Cs{} is simple, and the automorphism extends to an \emph{inner} automorphism on its bi-dual, then it is already inner on the \Cs{} itself. 

Part (i) of Theorem~\ref{thm:A} gives a \emph{necessary} condition on an automorphism on a \Cs{} to $A$ have \SAvP, and part (iii) gives a \emph{sufficient} condition for \SAvP. It is shown in Example~\ref{ex:B} that the sufficient condition is not necessary in general. In the other direction, we do not know if the necessary condition of (i) is also  sufficient in general (it probably isn't), but we do show that it is sufficient in three important cases: when the \Cs{} $A$ has stable rank one, Corollary~\ref{cor:B}, when $A$ is commutative, in which case \SAvP{} is equivalent to freeness, Corollary~\ref{cor:C}, and when $A$ is simple, Corollary~\ref{cor:A}. In the latter case \SAvP{} is equivalent to outerness when $A$ has no tracial state, and it is equivalent to proper outerness of the extension of the automorphism to the finite part of the bi-dual, when $A$ does admit a tracial state. 

Section~\ref{sec:applications} of the paper contains three applications of the strong averaging property of automorphisms, or rather of actions of a discrete group on a \Cs{} by automorphisms with this property.  We provide shorter and more direct proofs of existing results, and also extend these results in different directions. Given an action $\alpha \colon \Gamma \curvearrowright A$ with \SAvP{} of a discrete group $\Gamma$ on a unital \Cs{} $A$. Then we have:
\begin{itemize}
\item Each $\alpha$-invariant tracial state on $A$ has a \emph{unique} extension to a state on $A \rtimes_r \Gamma$ which is invariant under conjugation by unitaries from $A$. Ursu has recently in \cite{Ursu:trace} completely characterized actions for which invariant traces have unique extensions to traces on the crossed product. 
\item If $A$ has the Dixmier property, then $A \subseteq A \rtimes_r \Gamma$ has the relative Dixmier property. This was shown by Popa in \cite{Popa:JFA2000} when $A$ is simple, and this result was his motivation for considering the strong averaging property for automorphisms.
\item If $\Lambda$ is a subgroup of $\Gamma$ so that the action $\Lambda \curvearrowright A$ is minimal, then $A \rtimes_r \Lambda \subseteq A \rtimes_r \Gamma$ is $C^*$-irreducible, cf.\ \cite{Ror:C*-irreducible}, i.e., all intermediate \Cs s are simple. This is a classical result when $\Lambda$ is the trivial group, see \cite{Ror:C*-irreducible}, and was extended to normal subgroups $\Lambda \lhd \Gamma$ in \cite{BedosOmland:irreducible}. 

Moreover, if $\Lambda$ is normal in $\Gamma$, then there is a one-to-one Galois correspondance between intermediate \Cs s $A \rtimes_r \Lambda \subseteq D \subseteq A \rtimes_r \Gamma$  and intermediate groups $\Lambda \subseteq \Upsilon \subseteq \Gamma$, via $D = A \rtimes_r \Upsilon$. This result is already known (even for outer actions), cf.\ Cameron-Smith, \cite{CamSmith:Galois} (when $\Lambda$ is trivial) and Bedos-Omland, \cite{BedosOmland:irreducible}, using \cite{CamSmith:Galois}. We give  here an elementary self-contained proof using the strong averaging property.
\end{itemize}
.

 \noindent {\bf{Acknowledgements:}} I thank Erik Bedos, George Elliott, Thierry Giordano, Tron Omland, and Sorin Popa for useful comments and discussions.

\section{Popa's averaging property for automorphisms} \label{sec:SAvP}

\noindent We begin by formally defining the averaging properties, implicitly defined in Popa's paper \cite{Popa:JFA2000}, that is the topic of this paper. 

\begin{definition} \label{def:AvP}
Let $\cA$ be a unital \Cs, let $\cU(\cA)$ denote the group of unitary elements in $\cA$, and let $\alpha$ be an automorphism on $\cA$. 

Let $b \in \cA$. Denote by $C_\cA^{\, \alpha}(b)$, $C_\cA(b) = C_\cA^{\, \mathrm{id}_\cA}(b)$,  and $C_\cA^{\, \alpha} = C_\cA^{\, \alpha}(1_A)$ the norm closed convex hull of the sets
$$\{vb\alpha(v)^* : v \in \cU(\cA)\}, \qquad \{vbv^* : v \in \cU(\cA)\},  \qquad  \{v\alpha(v)^* : v \in \cU(\cA)\},$$ respectively. We use the same notation as above in the case where $b$ belongs to a \Cs{} $\cB$ that contains $\cA$.

We say that $\alpha$ has the \emph{averaging property} \AvP{} if $0 \in C_\cA^{\, \alpha}$, and that $\alpha$ has the \emph{strong averaging property} \SAvP{} if $0 \in C_\cA^{\, \alpha}(b)$, for all $b \in \cA$. 
\end{definition}

\noindent It is easy to see that if $b' \in C_\cA^{\, \alpha}(b)$, then $C_\cA^{\, \alpha}(b') \subseteq C_\cA^{\, \alpha}(b)$. Recall also that the usual Dixmier property of a \Cs{} $A$ holds when $C_\cA(b) \cap \C \cdot 1_\cA \ne \emptyset$, for all $b \in \cA$.\footnote{This definition of the (usual) Dixmier property is more strict than the commonly used one that just requires that $C_\cA(b)$ has non-empty intersection with the center of $\cA$. } It was shown in \cite{Haagerup-Zsido} that a simple unital \Cs{} has the Dixmier property if and only if it has at most one tracial state. See \cite{ArchRobTiku:Dixmier} for results about the Dixmier for non-simple \Cs s. 

Repeated application of the averaging procedure yields:

\begin{lemma} \label{lm:A}
If $\alpha_1, \dots, \alpha_n$ are automorphisms on a unital \Cs{} $A$ with \SAvP{} and $b_1, \dots, b_n \in A$, then for each $\ep >0$ there exist unitaries $v_1, \dots, v_m \in A$ such that
$$\Big\| \frac{1}{m} \sum_{i=1}^m v_i b_j \alpha_j(v_i)^*\Big\| < \ep, \quad j=1,2, \dots, n.$$
\end{lemma}

\noindent
The property described in Lemma~\ref{lm:A} above was introduced by Popa in  
\cite[Corollary 4.1]{Popa:JFA2000} as a way of establishing that the inclusion of a \Cs{} with the Diximier property inside its crossed product with a suitably outer action has the relative Dixmier property, cf.\ \cite{Popa:ENS1999}. The relative Dixmier property will be discussed further in Section~\ref{sec:applications}.

We note the following fact, whose  easy proof is  left to the reader. As usual, for each unitary element $u \in A$, $\mathrm{Ad}_u$ denotes the inner automorphism $a \mapsto uau^*$, $a \in A$.

\newpage

\begin{lemma} \label{lm:B} Let $\alpha$ be an automorphism on a unital \Cs{} $\cA$ and let $u \in \cU(\cA)$. Then:
\begin{enumerate}
\item $\Ad_u \circ \alpha$ has \AvP{} if and only if $0 \in C_\cA^{\, \alpha}(u)$,
\item $\Ad_u$ has \AvP{} if and only if $0 \in C_\cA(u)$.
\end{enumerate}
\end{lemma}

\begin{proposition} \label{prop:sr1}
Let $\alpha$ be an automorphism on a unital \Cs{} $\cA$. If $\alpha$ has \SAvP, then $\mathrm{Ad}_u \circ \alpha$ has \AvP, for all $u \in \cU(A)$. Conversely, if the stable rank of $A$ is one, and if $\mathrm{Ad}_u \circ \alpha$ has \AvP, for all $u \in \cU(A)$, then $\alpha$ has \SAvP.
\end{proposition}

\begin{proof} The first part follows immediately from Lemma~\ref{lm:B} (i). Suppose that $A$ has stable rank one and that $\mathrm{Ad}_u \circ \alpha$ has \AvP, for all $u \in \cU(A)$. Let $b \in A$. It follows from \cite{Ror:Annals} that there exist unitaries $u_1, u_2, u_3 \in A$ such that 
$$b = \frac{\|b\|}{3} \big( u_1+u_2+u_3\big) =  \frac{\|b\|}{3} u_1 + b_0, \qquad b_0 = \frac{\|b\|}{3} \big( u_2+u_3\big).$$
Since $0 \in C_A^\alpha(u_1)$ by Lemma~\ref{lm:B} (i), we conclude that $C_A^\alpha(b)$ contains elements arbitrarily close to $C_A^\alpha(b_0)$, and in particular, that $\inf \{ \|x\| : x \in C_A^\alpha(b)\} \le \|b_0\| = \frac23 \|b\|$. Repeating this process yields the result.
\end{proof}

\begin{question} \label{q:sr1} Does the reverse implication of Proposition~\ref{prop:sr1} hold for general \Cs s not necessarily of stable rank one?
\end{question}

\noindent One can relate the averaging properties of an automorphism on a \Cs{} to the usual Diximier property relatively to the crossed product \Cs{} as follows:

\begin{proposition} \label{prop:0} Let $\cA$ be a unital \Cs, and let $\alpha$ be an automorphism on $\cA$. Let $u \in \cA \rtimes_r \Z$ be the unitary implementing the action of $\alpha$. Then:
\begin{enumerate}
\item $\alpha$ has \AvP{} if and only if $0 \in C_{\cA}(u)$,
\item $\alpha$ has \SAvP{} if and only if $0 \in C_{\cA}(bu)$, for all $b \in \cA$.
\end{enumerate}

\end{proposition}

\begin{proof} For each $b \in \cA$ and for each unitary $v \in \cU(\cA)$ we have
$vb\alpha(v)^* = (vbuv^*)u^*$, which shows that $C^{\, \alpha}_\cA(b) = C_A(bu)u^*$, and hence that $C^{\, \alpha}_\cA = C_\cA(u)u^*$. 
\end{proof}

\noindent Recall that the bi-dual $A^{**}$ of a \Cs{} $A$ has a natural structure of a von Neumann algebra, and that it moreover is the universal enveloping von Neumann algebra of $A$, see \cite[Section 1.4]{Brown-Ozawa}.

\begin{definition} For each automorphism $\alpha$ on a unital \Cs{} $\cA$, let $\bar{\alpha}$ denote its  (unique) normal extension to $\cA^{**}$. 

Write $\cA^{**} = \cA^{**}_\fin \oplus \cA^{**}_\Inf$, where $ \cA^{**}_\fin$ and $\cA^{**}_\Inf$ denote the finite, respectively, the properly infinite central parts of $\cA^{**}$. Let $\bar{\alpha}_\fin$ and $\bar{\alpha}_\Inf$ denote the restrictions of $\bar{\alpha}$ to $ \cA^{**}_\fin$ and $\cA^{**}_\Inf$, respectively.
\end{definition}

\begin{definition} \label{def:prop-outer}
Following Elliott, \cite{Elliott:PRIM-1980}, an automorphism $\alpha$ on a \Cs{} $\cA$ is \emph{properly outer} if its restriction to any non-zero $\alpha$-invariant ideal $I$ of $A$ has distance 2 to the multiplier inner automorphisms on $I$.

For each $\alpha$-invariant ideal $I$ of $A$ we denote by $\overset{\centerdot}{\alpha}$ the ``descended'' automorphism on $\cA/I$ (when the ideal $I$ is clear from the context).

If  $\overset{\centerdot}{\alpha} \in \mathrm{Aut}(\cA/I)$ is outer, respectively, properly outer, for each $\alpha$-invariant ideal $I$ in $\cA$, then $\alpha$ is said to be \emph{residually outer}, respectively, \emph{residually properly outer}.

An automorphism $\beta$ on a von Neumann algebra $M$ is  \emph{properly outer} if its restriction  to each invariant central summand of $M$ is outer. 
\end{definition}

\noindent  It is a consequence of Sakai's theorem (that all bounded derivations on a simple \Cs{} are inner), \cite{Sakai:inner}, that any outer automorphism on a simple unital \Cs{} is properly outer.

Connes proved in \cite{Connes:Aut} that an automorphism $\beta$ on a von Neumann algebra $M$ is properly outer if and only if for each non-zero projection $p \in M$ and each $\ep>0$ there exists a non-zero projection $q \le p$ in $M$ such that $\|q\beta(q)\| < \ep$. We shall need this result in the version described in the lemma below. 

Proper outerness for an automorphism on a separable \Cs{} $A$ was recasted in eleven different ways in 
 \cite[Theorem 6.1]{OlePed:C*-dynamicIII}. One of the eleven conditions says, in analogy with Connes' result above, that for each non-zero hereditary sub-\Cs{} of $A$ and for each $\ep >0$ there exists a positive element $h$ in the hereditary subalgebra such that $\|h\|=1$ and $\|h\alpha(h)\| < \ep$. One can further choose $h$ so that, for any given $b \in A$, we have $\|hb\alpha(h)\| < \ep$, cf.\ \cite[Lemma 7.1]{OlePed:C*-dynamicIII} or \cite{Kis:Auto}.
 
Although a priori not obvious, the definition of proper outerness  for \Cs s, when applied to an automorphism on a von Neumann algebra, agrees with the definition of proper outerness for von Neumann algebras. (One can use Connes' characterization of proper outerness to see this.)

\begin{lemma} \label{lm:D0}
Let $M$ be a von Neumann algebra and let $\alpha$ be a properly outer automorphism on $M$. 
\begin{enumerate}
\item For each $b \in M$, for each $\ep>0$, and for  non-zero projection $p$ in $M$  there exists a non-zero subprojection $q$ of $p$ such that $\|qb\alpha(q)\|<\ep$. 
\item For each $b \in M$ and for each $\ep >0$  there exists a family $(p_i)_{i \in I}$ of pairwise orthogonal projections in $M$ such that $\sum_{i \in I} p_i = 1$ and $\|p_i b \alpha(p_i)\| < \ep$, for all $i \in I$.
\end{enumerate}
\end{lemma}

\begin{proof} (i). We first note that Connes' characterization of properly outer automorphisms implies that 
 (i) holds when $b$ is a unitary. Indeed, if $u \in M$ is a unitary, then $\beta = \mathrm{Ad}_u \circ \alpha$ is still properly outer, so given $p$ and $\ep$ there exists a non-zero projection $q$ below $p$ such that $\|qu\alpha(q)\| = \|q\beta(q)u\| = \|q\beta(q)\| < \ep$. 

We can write $b = C \cdot (u_1+u_2+u_3)$, for some constant $C > 0$ (which can taken to be $C= \frac23 \|b\|$), and for unitaries $u_1, u_2, u_3 \in M$. By the argument above there exist non-zero projections $q:= q_3 \le q_2 \le q_1 \le p$ such that $\|q_ju_j\alpha(q_j) \| < \ep/3C$, for all $j$. It follows that $\|q u_j \alpha(q)\| < \ep/3C$, for all $j$, which in turns implies that $\|qb\alpha(q)\| < \ep$, as desired. 

(ii) follows from (i) via a standard maximality argument. 
\end{proof}

\noindent
It follows easily from Lemma~\ref{lm:A} that any automorphism with \SAvP{} must be outer. More is true:

\begin{proposition} \label{prop:A1} 
Any automorphism on a unital \Cs{} with \SAvP{}  is residually outer.
\end{proposition}

\begin{proof} Let $\alpha$ be an automorphism on a unital \Cs{} $\cA$ which is not residually outer. Let $I$ be a proper closed two-sided invariant ideal in $\cA$ for which $\overset{\centerdot}{\alpha} \in \mathrm{Aut}(\cA/I)$ is inner, say equal to $\Ad_u$, for some unitary $u \in \cA/I$.  Let $b \in \cA$ be a lift of $u$ and let $\pi \colon \cA \to \cA/I$ denote the quotient mapping. Then, for each $v \in \cU(\cA)$, we have $\pi(vb^*\alpha(v)^*) = \pi(v) u^* \overset{\centerdot}{\alpha}(\pi(v))^* = \pi(v)\pi(v)^*u^* = u^*$. Hence $\pi(x) = u^*$, for all $x \in C_A^{\, \alpha}(b^*)$, which shows that $\alpha$ cannot have \SAvP.
\end{proof}

\noindent The lemma below is a consequence of the universal property of the bi-dual of a \Cs:

\begin{lemma} \label{lm:A1} For an automorphism $\alpha$ on a unital \Cs{} $\cA$, $\bar{\alpha}_\fin$ is properly outer if and only if, for each $\alpha$-invariant tracial state $\tau$ on $\cA$, the normal extension of $\alpha$ to $\pi_\tau(\cA)''$ is outer.
\end{lemma}

\noindent We have arrived at the second obstruction to having \SAvP:

\begin{proposition} \label{prop:B1}
If $\alpha$ is an automorphism with \SAvP{} on a unital \Cs{} \cA, then $\bar{\alpha}_\fin$ is properly outer.
\end{proposition}

\begin{proof} Suppose that $\bar{\alpha}_\fin$ is not properly outer. Then, by Lemma~\ref{lm:A1}, there is an $\alpha$-invariant tracial state $\tau$ on $\cA$ such that $\alpha$ extends to an inner automorphism on  $\pi_\tau(\cA)''$, i.e., there is a unitary $u \in \pi_\tau(\cA)''$ such that $u \pi_\tau(b) u^* = \pi_\tau(\alpha(b))$, for all $b \in \cA$. Approximating $u$ by unitaries from $\pi_\tau(\cA)$ with respect to $\| \, \cdot \, \|_{2,\tau}$, we find a unitary $w \in \pi_\tau(\cA)$ with $\bar{\tau}(w^*u) \ne 0$, where $\bar{\tau}$ is the extension of $\tau$ to $\pi_\tau(\cA)''$. Choose $b \in \cA$ such that $\pi_\tau(b) = w$. Then, for each unitary $v \in \cA$, we have 
$$\bar{\tau}(\pi_\tau(vb^*\alpha(v)^*)u) = \bar{\tau}(\pi_\tau(v)w^* u\pi_\tau(v^*)) = \bar{\tau}(w^*u).$$
Hence $\bar{\tau}(\pi_\tau(x)u) = \bar{\tau}(w^*u) \ne 0$, for all $x \in C_\cA^{\, \alpha}(b^*)$, which proves that $0 \notin C_\cA^{\, \alpha}(b^*)$, so $\alpha$ does not have \SAvP.
\end{proof}

\noindent The result below was shown by Popa in \cite[Corollary 4.4]{Popa:ENS1999} in the case where $b=1$.

\begin{lemma} \label{lm:D} Let $\alpha$ be a properly outer automorphism on a von Neumann algebra $M$. Then $0$ belongs to the strong operator closure of $C_M^{\, \alpha}(b)$, for all $b \in M$.
\end{lemma}

\begin{proof} Fix $\ep >0$ and choose a family  $(p_i)_{i \in I}$  of projections in $M$ with $\sum_i p_i = 1$ and $\|p_ib\alpha(p_i)\| < \ep$, for all $i$, cf.\ Lemma~\ref{lm:D0}.  For each finite subset $F \subseteq I$ and $\omega = (\omega_i)_{i \in F} \in \T^F$, set $p_F = \sum_{i \in F} p_i$ and $u_{F,\omega} = \sum_{i \in F} \omega_i p_i + (1-p_F)$. Then
$$
y_F:= \int_{\T^F} u_{F,\omega} \, b \, \alpha(u_{F,\omega})^* \, d\mu(\omega) = \sum_{i \in F} p_ib\alpha(p_i) + (1-p_F)x(1-\alpha(p_F)),
$$
belongs to $C_M^\alpha(b)$, 
where $\mu$ is the Haar measure on $\T^{F}$. The sum $y:=  \sum_{i \in I} p_ib\alpha(p_i)$ belongs to the strong operator closure of $\{y_F\}_F$, and hence to the strong operator closure of  $C_M^{\, \alpha}(b)$. 
 By orthogonality of the projections $\{p_i\}_{i \in I}$, we deduce that $\|y\| = \sup\{\|p_ib\alpha(p_i)\|\} \le \ep.$ 
\end{proof}

\noindent Recall the well-known---and often used---fact (which follows from basic properties of the bi-dual and a Hahn-Banach argument): If $C$ is a bounded convex subset of a \Cs{} $\cA$, then $\overline{C} = \overline{C}^w \cap \cA$, where
$\overline{C}$ and $\overline{C}^w$ denote the norm, respectively, the strong (or, equivalently, the weak) operator closure of $C$ in $\cA^{**}$.

\begin{lemma} \label{lm:E} Let $A$ be a  unital \Cs, let $\alpha$ be an  automorphism on $\cA$, and let $\bar{\alpha}$ be its extension  to the bi-dual $A^{**}$.
\begin{enumerate}
\item If $0$ belongs to the strong operator closure of $C_{A^{**}}^{\; \bar{\alpha}}$, then $\alpha$ has \AvP.
\item If $0$ belongs to the strong operator closure of $C_{A^{**}}^{\; \bar{\alpha}}(b)$, for all $b \in A$, then $\alpha$ has \SAvP.
\end{enumerate}
\end{lemma}

\begin{proof} This follows from the remark above applied to $C = \mathrm{conv}\{v \alpha(v)^* : v \in \cU(\cA)\}$, respectively, to  $C(b) = \mathrm{conv}\{vb \alpha(v)^* : v \in \cU(\cA)\}$, and the fact that $v \bar{\alpha}(v)^*$ and $vb \bar{\alpha}(v)^*$  belong to the weak operator closure of $C$, respectively, $C(b)$, for all $v \in \cU(\cA^{**})$. 
\end{proof}

\noindent As an immediate corollary to Lemmas \ref{lm:D} and \ref{lm:E} we obtain:

\begin{corollary} \label{cor:A} Let $\alpha$ be an automorphism on a unital \Cs{} $\cA$. If $\bar{\alpha}$ is properly outer, then $\alpha$ has \SAvP. 
\end{corollary}

\noindent The lemma below can also be derived from \cite[Theorem 4.7]{NgRobSkou:TAMS}. The special case under consideration here is quite easy, and we include its proof for completeness of the exposition.

\begin{lemma} \label{lm:G1} Let $A$ be a unital \Cs, let $a \in A$, let $\ep >0$, and let $p \in A$ be a properly infinite projection satisfying $\|pap\| \le \ep$ and $1-p \precsim p$. Then 
$$\inf \{ \|x\| : x \in C_A(a)\} \le \ep.$$
\end{lemma}

\begin{proof} Observe first that for each finite partition $1 = \sum_{i=1}^n p_i$ of the unit into pairwise orthogonal projections $p_i \in A$ and for each $a \in A$, we have $\sum_{i=1}^n p_iap_i \in C_A(a)$.  This is obtained by averaging with conjugates of $a$ by unitaries of the form $\sum_{i=1} \omega_i p_i$, for $\omega_i \in \T$.

Let $n \ge 1$. Since $p$ is properly infinite and  $1-p \precsim p$ we can find $n$ pairwise orthogonal subprojections of $p$ each of which is equivalent to $1-p$. In other words, 
we can  partition the unit of $A$ into projections $p_0,p_1, \dots, p_{n+1} = 1-p$ in $A$  such that $p_1 \sim p_2 \sim \cdots \sim p_{n+1}$. By the observation above we can find $b \in C_A(a)$ with $p_ibp_j = 0$, when $i \ne j$, and $p_ibp_i = p_iap_i$. Note that $\|p_iap_i\| \le \|pap\| \le \ep$, when $0 \le i \le n$. For each permutation $\sigma$ of $\{0, 1,2, \dots, n+1\}$ fixing $0$, let $u_\sigma$ be a unitary in $A$ satisfying $u_\sigma p_i u_\sigma^* = p_{\sigma(i)}$, for all $i$.
Set $$c = \frac{1}{(n+1)!} \sum_{\sigma} u_\sigma b u_{\sigma}^* \in C_A(a).$$ Then $p_icp_j = 0$, when $i \ne j$; and since $p_i u_\sigma b u_\sigma^*p_i = u_\sigma p_{\sigma^{-1}(i)} b p_{\sigma^{-1}(i)} u_\sigma^*$, we get
$$\|c\| = \max\{p_i c p_i : 0 \le i \le n+1\} \le \frac{n\ep + \|a\|}{n+1}.$$
Since $n \ge 1$ was arbitrary this proves the lemma.
\end{proof}

%\newpage

\begin{lemma} \label{lm:I} If $\alpha$ is a residually outer automorphism on a unital \Cs{} $\cA$, then:
\begin{enumerate}
\item $\|\alpha - \mathrm{id}_\cA \|= 2$,
\item for each $\ep>0$, there exists a positive element $h \in \cA$ with $\|h\|=1$ and $\|h\alpha(h)\|\le \ep$. 
\end{enumerate}
\end{lemma}

\begin{proof} (i). If $\|\alpha - \mathrm{id}_\cA\| < 2$, then $\alpha$ is the exponential of a bounded derivation, \cite[Theorem 7]{Kadison-Ringrose}, see also \cite[8.7.7]{Ped:auto}, and approximately inner, which again implies that $\alpha$ is ideal preserving. Take a maximal proper closed two-sided ideal $I$ in $\cA$, and consider the automorphism $\overset{\centerdot}{\alpha}$ on the simple unital \Cs{} $\cA/I$. Then $\overset{\centerdot}{\alpha}$ is still the exponential of a bounded derivation, so by Sakai's theorem, \cite{Sakai:inner}, it is inner. Hence $\alpha$ is not residually outer.

(ii) follows from (i) and the Olesen--Pedersen ``sine-cosine'' theorem, \cite[Theorem 5.1]{OlePed:C*-dynamicIII}.
\end{proof}

\noindent We are now ready to characterize  automorphisms with the strong averaging property.

\begin{theorem} \label{thm:A} Let $\alpha$ be an automorphism on a unital \Cs{} $A$. In parts (ii) and (iii) assume further that $A$ is separable. If $A$ has no tracial state, then $A^{**}_\fin$ is zero in which case we declare the condition below on $\bar{\alpha}_\fin$ to be vacuously satisfied.
\begin{enumerate}
\item If $\alpha$ has \SAvP, then $\alpha$ is  residually outer and  $\bar{\alpha}_\fin$ is properly outer.
\item Conversely, if $\alpha$ is residually outer and $\bar{\alpha}_\fin$ is properly outer, then $\Ad_u \circ \alpha$ has \AvP, for all unitaries $u \in A$.
\item If $\alpha$ is residually properly outer and $\bar{\alpha}_\fin$ is properly outer, then $\alpha$ has \SAvP.
\end{enumerate}
\end{theorem}

\begin{proof} (i) has already been established in Propositions~\ref{prop:A1} and \ref{prop:B1}.

If $\alpha$ is residually outer and $\bar{\alpha}_\fin$ is properly outer, then so is $\Ad_u \circ \alpha$, for all unitaries $u \in A$. Hence, to prove (ii), it suffices to show that $\alpha$ has \AvP, under the assumptions in (ii).

We combine the proofs of (ii) and (iii), since they are similar. For each $\bar{\alpha}$-invariant central projection $q$ in $A^{**}$, consider the restriction $\bar{\alpha}_q$ of $\bar{\alpha}$ to $qA^{**}$, and let $\pi_q \colon A \to qA^{**}$ be the natural \sh{} arising as the inclusion $A \subseteq A^{**}$ composed with the quotient mapping $A^{**} \to qA^{**}$. 

The strategy is to prove that $\bar{\alpha}$ has the SOT-\AvP, respectively, the SOT-\SAvP, by which we mean that $0$ belongs to the strong operator closure of $C_{A^{**}}^{\, \bar{\alpha}}$, respectively, of $C_{A^{**}}^{\, \bar{\alpha}}(b)$, for all $b \in A$, and then apply Lemma~\ref{lm:E}. 

We assume that $\bar{\alpha}_\fin$ is properly outer, which is common for (ii) and (iii). Accordingly, we can write $\cA^{**} = q \cA^{**} + (1-q) \cA^{**}$, with $q$ an $\bar{\alpha}$-invariant central projection in $\cA^{**}$ such that $\bar{\alpha}_{1-q}$ is properly outer and $\bar{\alpha}_q = \Ad_w$ for some unitary $w \in  q \cA^{**}$. By the assumption that $\bar{\alpha}_\fin$ is properly outer, either $q=0$ or $q \cA^{**}$ is properly infinite. If $q=0$, then $\bar{\alpha}$ is SOT-\SAvP{} by Corollary~\ref{cor:A}, and we are done. 

We know from Lemma~\ref{lm:D} that $\bar{\alpha}_{1-q}$ is SOT-\SAvP, and hence SOT-\AvP. The problem is therefore reduced to the properly infinite central summand $qA^{**}$, i.e., to show that $\bar{\alpha}_{q} = \Ad_w$ is SOT-\AvP, respectively, SOT-\SAvP{} as an automorphism on $qA^{**}$. Fix $\ep >0$. By a standard maximality argument it suffices to show that for each non-zero central subprojection $q_0$ of $q$ there exists a non-zero central subprojection $q_1$ of $q_0$ such that 
$$\inf\{\|x\| :x \in C_{q_1A^{**}}^{\, \bar{\alpha}_{q_1}}\} \le \ep, \quad \text{respectively,} \; \;
\inf\{\|x\| :x \in C_{q_1A^{**}}^{\, \bar{\alpha}_{q_1}}(\bar{b}q_1)\} \le \ep, \; b \in A,$$
where $\bar{b}$ is the image of $b$ in $A^{**}$. Note that
$$C_{q_1 A^{**}}^{\, \bar{\alpha}_{q_1}} =C_{q_1A^{**}}(w_1)w_1^* \quad \text{and}
\quad C_{q_1A^{**}}^{\, \bar{\alpha}_{q_1}}(\bar{b}q_1) = C_{q_1A^{**}}(w_1 \bar{b}q_1)w_1^*,$$
where $w_1=wq_1$.

The assumption in (ii) and (iii) that $A$ is separable implies that $\kappa_{A^{**}} \le \aleph_0$, see, eg., \cite[Definition 5.5 and Lemma 5.6]{ORT}. Recall that $\kappa_M$ is the largest cardinality of a set of pairwise equivalent and pairwise orthogonal non-zero projections in a von Neumann algebra $M$. If $\kappa_M \le \aleph_0$, then any two properly infinite projections in $M$ are Murray von Neumann equivalent if and only if they have the same central support.

Assume now that $\alpha$ is residually outer. Let $\ep >0$ be as given above. By Lemma~\ref{lm:I}, applied to the quotient $\pi_{q_0}(A)$ of $A$, there exists $h \in \pi_{q_0}(A)^+$ such that $\|h\|=1$ and $\|h \bar{\alpha}_{q_0}(h)\| \le \ep/3$. Choose continuous functions $g_1,g_2$ on $[0,1]$ such that $0 \le g_2 \le g_1 \le 1$, $g_2(1)=1$, $g_1g_2=g_2$ and $\mathrm{supp}(g_1) \subseteq [1-\ep/3,1]$. Choose a non-zero projection $p$ in the closed hereditary sub-\Cs{} of $q_0A^{**}$ generated by $g_2(h)$. Then $pg_1(h)=p$ and $\|g_1(h)(1-h)\| \le \ep/3$, which implies that $\|p(1-h)\| \le \ep/3$.  It follows that 
$$\|pw_0p\| =\|p \bar{\alpha}_{q_0}(p)w_0\| = \|p \bar{\alpha}_{q_0}(p)\| \le  \|p \bar{\alpha}_{q_0}(p) - p h\bar{\alpha}_{q_0}(hp)\| + \ep/3 \le  \ep.$$
Let $q_1$ be the central support of $p$. Then $q_1-p \precsim p$ (since $p$ is a properly infinite projection whose central support dominates the central support of $q_1-p$). Moreover, $\|pw_1p\| \le \ep$ (with $w_1=wq_1$), so it follows from Lemma~\ref{lm:G1}, applied to the unital \Cs{} $q_1A^{**}$, that 
$$\inf\{\|x\| :x \in C_{q_1A^{**}}^{\, \bar{\alpha}_{q_1}}\} = \inf\{\|x\| :x \in C_{q_1A^{**}}(w_1)\} \le \ep,$$
as desired. This proves (ii).

To prove (iii) we assume that $\alpha$ is residually properly outer. Given $b \in A$, which we may assume to be a contraction, and with $\ep$ as given above, we can find $h \in \pi_{q_0}(A)^+$ such that $\|h\|=1$ and $\|h \bar{b} q_0 \bar{\alpha}_{q_0}(h)\| \le \ep/3$, cf.\ the comments below Definition~\ref{def:prop-outer}, where $\bar{b}$ as above is the image of $b$ in $A^{**}$. Choose the projection $p$ as in the previous paragraph (with respect to the same functions $g_1,g_2$, now applied to the present $h$). Then, arguing as above, 
$$\|p\bar{b}q_0w_0p\| = \|p \bar{b}q_0 \bar{\alpha}_{q_0}(p)\| \le  \|p \bar{b}q_0\bar{\alpha}_{q_0}(p) - p hb_0\bar{\alpha}_{q_0}(hp)\| + \ep/3 \le  \ep.$$
Take again $q_1$ to be the central support of $p$. Then $\|p\bar{b}q_1 w_1p\| \le \ep$ and $q_1-p \precsim p$, so Lemma~\ref{lm:G1} yields that
$$\inf\{\|x\| :x \in C_{q_1A^{**}}^{\, \bar{\alpha}_{q_1}}(\bar{b}q_1)\} = \inf\{\|x\| :x \in C_{q_1A^{**}}(\bar{b}q_1w_1)\} \le \ep,$$
as desired. This proves (iii).
\end{proof}

\noindent In the cases where the \Cs{} $A$ in Theorem~\ref{thm:A} above is either simple, of stable rank one, or commutative, one can improve the statement of the theorem as in the corollaries below. Our first corollary addresses the case where $A$ is simple:

\begin{corollary} \label{cor:A}  Let $A$ be a simple separable unital \Cs, and let $\alpha$ be an automorphism on $A$. 
\begin{enumerate}
\item If $A$ admits a tracial state, then $\alpha$ has \SAvP{} if and only if $\bar{\alpha}_\fin$ is properly outer on $A^{**}_\fin$.
\item If $A$ does not have a tracial state,  then $\alpha$ is \SAvP{} if and only if $\alpha$ is outer on $A$.
\end{enumerate}
\end{corollary}

\begin{proof} The ``only if'' parts of (i) and (ii) both follow from Theorem~\ref{thm:A}~(i). If $A$ has a tracial state and $\bar{\alpha}_\fin$ is properly outer, then $\alpha$ is outer (hence properly outer, hence residually properly outer, because $A$ is simple), so it has \SAvP{} by Theorem~\ref{thm:A}~(iii). If $A$ has no trace, then the condition on $\bar{\alpha}_\fin$ becomes vacuous, and by the same argument as above we conclude that that every outer automorphism on $A$ has \SAvP.
\end{proof}

\begin{corollary} \label{cor:B} For an automorphism $\alpha$ on a unital separable \Cs{} $A$ of stable rank one, the following are equivalent:
\begin{enumerate}
\item $\alpha$ has \SAvP,
\item $\mathrm{Ad}_u \circ \alpha$ has \AvP, for all unitaries $u \in A$,
\item $\alpha$ is residually outer and $\bar{\alpha}_\fin$ is properly outer.
\end{enumerate}
\end{corollary}

\begin{proof} (iii) $\Rightarrow$ (ii) is Theorem~\ref{thm:A}  (ii); (ii) $\Leftrightarrow$ (i) is Proposition~\ref{prop:sr1}, and (i) $\Rightarrow$ (iii) is Theorem~\ref{thm:A} (i).
\end{proof}

\noindent The (strong) averaging property is equivalent to freeness on commutative \Cs s:

\begin{corollary} \label{cor:C} 
Let $X$ be a compact Hausdorff space and consider the unital \Cs{} $\cA = C(X)$. The following conditions are equivalent for each automophism $\alpha$ on $\cA$.
\begin{enumerate}
\item $\alpha$ has \AvP,
\item $\alpha$ has \SAvP,
\item the induced homeomorphism $\hat{\alpha}$ of $\alpha$ on $X$ is free,
\item the extension $\bar{\alpha}$ of $\alpha$ to $A^{**}$ is properly outer.
\end{enumerate}
\end{corollary}

\noindent
Since $A^{**}$ is abelian, (iv) just says that $\bar{\alpha}$ is not the identity on any central summand of $A^{**}$. The proof below does not use the full force of Theorem~\ref{thm:A}.

\begin{proof} (ii) $\Rightarrow$ (i) holds trivially (and (i) $\Rightarrow$ (ii) holds when $\cA$ is abelian). (iv) $\Rightarrow$ (i) follows from Lemmas~\ref{lm:D} and \ref{lm:E}.

(i) $\Rightarrow$ (iii). Suppose that $\hat{\alpha}$ is not free. Then there is a pure state $\rho$ on $\cA$ which is invariant under $\alpha$. Since pure states on abelian \Cs s are multiplicative, we get $\rho(v\alpha(v)^*) = 1$, for all unitaries $v$ in $\cA$. Hence $\alpha$ cannot have \AvP.

(iii) $\Rightarrow$ (iv).  Assume that $\bar{\alpha}$ is not properly outer. Then there is an invariant non-zero central projection $q \in \cA^{**}$ such that $\bar{\alpha}$ is inner on $\cA^{**}q$. Since $\cA^{**}$ is abelian, this implies that $\bar{\alpha}$ is the identity on $A^{**}q$. Consider the (canonical) \sh{} $\pi_q \colon \cA \to \cA^{**}q$. Then $\pi_q=\bar{\alpha} \circ \pi_q = \pi_ q \circ \alpha$. Take a pure state $\sigma$ on $\pi_q(\cA)$ and set $\rho = \sigma \circ \pi_q$, which is a pure state on $\cA$. Then $\rho \circ \alpha = \sigma \circ \pi_q \circ \alpha = \rho$. Hence $\hat{\alpha}$ is not free.
\end{proof}

\begin{example}  \label{ex:A} Let $A = \bigotimes_{n=1}^\infty M_{k_n}(\C)$ be a UHF-algebra with unique tracial state $\tau$, where $k_n \ge 2$, for all $n$. Consider the unitary elements $u_n = \diag(1,1, \cdots, 1, -1) \in M_{k_n}(\C)$, and the automorphism
$$\alpha = \bigotimes_{n=1}^\infty \Ad_{u_n}$$
on $A$. Then $\alpha$ is outer. But if $\sum_{n=1}^\infty \|1_{k_n} - u_n\|_2 < \infty$, then 
$\alpha$ extends to an inner automorphism $\Ad_u$ on $\pi_\tau(A)''$, with $u = \bigotimes_{n=1}^\infty u_n \in \pi_\tau(A)''$.  Hence $\alpha$ is not \SAvP, and if $\tau(u) \ne 0$, which is the case if $\sum_{n=1}^\infty 1/k_n < \infty$, then $\alpha$ does not have \AvP.

In this example, $\alpha$ is an outer automorphism on $A$ (even residually properly outer), but $\bar{\alpha}_\fin$ is inner on $A^{**}_\fin = \pi_\tau(A)''$. In particular, $\alpha$ does not have \SAvP.
\end{example}

\begin{example} \label{ex:B}
Let $2 \le n < \infty$ and let $\mathcal{T}_n$ be the Cuntz-Toeplitz algebra generated by isometries $s_j$, $1 \le j \le n$, with orthogonal range projections, described by the extension
$$ 0 \to \cK \to \mathcal{T}_n \to \cO_n \to 0,$$
where $\cK$ denotes the compact operators on a separable infinite dimensional Hilbert space.
Let us first note that no automorphism on $ \mathcal{T}_n$ is properly outer, since its restriction to $\cK$ always is multiplier inner. Secondly, there do exist residually outer automorphisms on  $\mathcal{T}_n$. Take, for example, any automorphism $\alpha$ on $\mathcal{T}_n$  that permutes the generators $s_j$ according to some permutation $\sigma \ne \mathrm{id}$ on $\{1,2, \dots, n\}$. Then $\alpha$ is outer on  $\mathcal{T}_n$ and on $\cO_n$, hence residually outer. 

We finally observe that an automorphism $\alpha$ on $ \mathcal{T}_n$ has \SAvP{} if and only if it is residually outer. This fact, that we will prove below, does not follow from Theorem~\ref{thm:A} above, and moreover shows that the sufficient condition in Theorem~\ref{thm:A} (iii) is not necessary. 

Take a residually outer automorphism $\alpha$ on $ \mathcal{T}_n$. Then $\alpha$ descends to an outer automorphism $\overset{\centerdot}{\alpha}$ on $\cO_n$, which has \SAvP{} by Corollary~\ref{cor:A} above. Let $b \in \mathcal{T}_n$ be given, and let $\pi$ denote the quotient mapping $\mathcal{T}_n \to \cO_n$. Let $\ep >0$. Then there exist unitaries $v_1, \dots, v_m \in \cO_n$ such that
$$\Big\| \frac{1}{m} \sum_{j=1}^m v_j \pi(b) \overset{\centerdot}{\alpha}(v_j)^*\Big\| < \ep.$$
The unitaries $v_j$ lift to unitaries $u_j \in \mathcal{T}_n$ (because the unitary group of $\cO_n$ is connected). It follows that there exists $x \in \cK$ such that
$$\Big\| \frac{1}{m} \sum_{j=1}^m u_j b \alpha(u_j)^*-x\Big\| < \ep.$$
The restriction of $\alpha$ to $\cK$ is multiplier inner, say equal to $\Ad_w$, for some $w \in \cM(\cK) = B(H)$. It follows from \cite[Theorem 4.7]{NgRobSkou:TAMS}, see also \cite[Theorem 2.1 and Proposition 2.2]{HjeRor:stable}, that $0 \in C_{\widetilde{\cK}}(y)$, for each $y \in \cK$, and in particular for $y=xw$. We can therefore find unitaries $z_1, \dots, z_n \in \widetilde{\mathcal{K}} \subseteq \mathcal{T}_n$ with
$$\Big\| \frac{1}{n} \sum_{i=1}^n z_i x \alpha(z_i)^*\Big\| =  \Big\| \frac{1}{n} \sum_{i=1}^n z_i xw z_i^*\Big\| < \ep.$$
In summary, we have
$$\Big\| \frac{1}{nm} \sum_{i,j} z_iu_j b \alpha(z_iu_j)^*\Big\| < 2\ep.$$
\end{example}

\noindent It was shown in Example~\ref{ex:B} that the sufficient condition Theorem~\ref{thm:A}~(iii), to ensure the strong averaging property of an automorphism, is not necessary. We do not know if the necessary condition in  Theorem~\ref{thm:A}~(i) is sufficient: 

\begin{question} \label{q:A1} Is it the case that an automorphism $\alpha$ on a unital \Cs{} $A$ has \SAvP{} if and only if it is residually outer and $\bar{\alpha}_\fin$ is properly outer?
\end{question}

\noindent The ``only if'' part is established in Theorem~\ref{thm:A}, and the example above shows that the condition ``residually properly outer'' is too strong. The situation in Example~\ref{ex:B} is special in that every unitary in any quotient of $\mathcal{T}_n$ lifts to a unitary in $\mathcal{T}_n$. This could indicate that the answer to the question above could be negative and that a complete characterization of automorphisms with \SAvP{} is strictly between the necessary condition of Theorem~\ref{thm:A}~(ii) and the sufficient condition of Theorem~\ref{thm:A}~(iii).

An affirmative answer to Question~\ref{q:sr1} implies an affirmative answer to Question~\ref{q:A1} by Theorem~\ref{thm:A} (ii).

\section{Applications} \label{sec:applications}

\noindent We give in this section three applications of the strong averaging property for an automorphism and for actions with this property. The applications provide simpler proof of existing results, which is a main purpose of advocating this property of automorphisms, but we are also able to improve some of these existing results in different directions. 

Given an  action $\alpha$ of a discrete group $\Gamma$ on a unital \Cs{} $A$. We say that $\alpha$ has \SAvP{} if $\alpha_t$ has \SAvP, for all $t \in \Gamma$, $t \ne e$, where $e \in \Gamma$ denotes the neutral element. The associated reduced crossed product is denoted $A \rtimes_r \Gamma$. Denote by $t \mapsto u_t$, $t \in \Gamma$, the unitary  representation of $\Gamma$ inside the crossed product, so that $A \rtimes_r \Gamma$ is the \Cs{} generated by $A \cup \{u_t : t \in \Gamma\}$. 

Let $E \colon A \rtimes_r \Gamma \to A$ denote the canonical conditional expectation, and for $t \in \Gamma$, let $E_t \colon  A \rtimes_r \Gamma \to A$ be given by $E_t(x) = E(xu_t^*)$, which is the ``Fourier coefficient'' of $x$ at $t$. Define the support, $\mathrm{supp}(x)$, of an element $x \in A \rtimes_r \Gamma$ to be the set $\{t \in \Gamma: E_t(x) \ne 0\}$. 

\begin{lemma} \label{lm:1}
Let $\alpha \colon \Gamma \curvearrowright A$ be a \SAvP{} action of a discrete group $\Gamma$ on a unital \Cs{} $A$. Let $x  \in A \rtimes_r \Gamma$ be given. 
\begin{enumerate}
\item If $E(x)$ is central in $A$, then $E(x) \in C_A(x)$.
\item If  $E_t(x)= 1_A$,  for some $t \in \Gamma$, then $u_t \in C_A^{\, \alpha_{t^{-1}}}(x)$.
\item $\mathrm{dist}(A, C_A(x)) = 0$.
\end{enumerate}
\end{lemma}

\noindent Note that $C_A^{\, \alpha_{s}}(x)$ is contained in $C^*(A,x)$, the smallest sub-\Cs{} of $A \rtimes_r \Gamma$ that contains $A$ and $x$, for all $x \in A \rtimes_r \Gamma$ and all $s \in \Gamma$.

\begin{proof} (i) and (ii). Let $t \in \Gamma$ and assume that $E_t(x)$ is central in $A$. (In part (i) we have $t=e$, $u_t = 1_A$, and $E_e(x) = E(x)$, while $E_t(x) = 1_A$ in part (ii).) Let $\ep >0$ and choose $x_0   \in A \rtimes_r \Gamma$ of finite support $F \subseteq \Gamma$ such that $\|x-x_0\| \le \ep$ and $E_t(x_0) = E_t(x)$. By Lemma~\ref{lm:A} we can find unitaries $v_1, \dots, v_m \in A$ such that
%$$\Big\| \frac{1}{m} \sum_{j=1}^m v_j E_s(x_0) \alpha_{st^{-1}}(v_j)^*\Big\| < \ep/|F|, \quad s \in F \setminus \{t\}.$$
%Then,
$$\big\| \frac{1}{m} \sum_{j=1}^m v_j E_s(x_0) u_s \alpha_{t^{-1}}(v_j)^*\big\|  = \big\| \sum_{j=1}^m v_j E_s(x_0) \alpha_{st^{-1}}(v_j)^*\big\| < \ep/|F|,$$
for $s \in F \setminus \{t\}$. For each $j$, $v_j E_t(x_0)u_t \alpha_{t^{-1}}(v_j)^* = v_j E_t(x_0) v_j^* u_t = E_t(x)u_t$. It follows that
$$\Big\| \frac{1}{m} \sum_{j=1}^m v_j x \alpha_{t^{-1}}(v_j)^*  -  E_t(x) u_t\Big\|  \le \Big\| \frac{1}{m} \sum_{j=1}^m v_j x_0 \alpha_{t^{-1}}(v_j)^*  -E_t(x)  u_t\Big\| + \|x-x_0\| < 2\ep.$$
This shows that (i) and (ii) hold.

(iii). Proceeding as above with $t=e$ and with $E(x_0) = E(x)$ (not necessarily assuming this to be central), we arrive at
$$\Big\| \frac{1}{m} \sum_{j=1}^m v_j x v_j^*  - \frac{1}{m} \sum_{j=1}^m v_j E(x) v_j^* \Big\| < 2\ep,$$
from which (iii) follows. \end{proof}

\noindent It is well-known that for an action $\alpha \colon \Gamma \curvearrowright A$ on a unital \Cs{} $A$, any tracial state $\tau$ on $A$ extends to a tracial state on $A \rtimes_r \Gamma$ if and only if it is invariant under the action $\alpha$, and that an extension in this case is given by $\bar{\tau} := \tau \circ E$. 

The question of when this extension is unique is subtle. Sufficient conditions for uniqueness of extending a tracial state on (non-commutative) $A$ to a tracial state on $A \rtimes_r \Gamma$ were given in \cite{Thomsen:trace} (in the case where $\Gamma$ is abelian) and a necessary and sufficient condition in the general case was recently obtained by Ursu in \cite[Theorem 1.8]{Ursu:trace}. We show below that
 the \SAvP{} condition  is sufficient to ensure unique trace extension, and it further implies a stronger uniqueness result. 
 
 The sufficient condition given in \cite{Thomsen:trace} involves outerness of the action on the von Neumann algebra completion of the \Cs{} with respect to the considered trace, a condition which is ensured if the action on $A^{**}_\fin$ is properly outer, which for simple unital \Cs s (with a trace) is equivalent to \SAvP, cf.\ Corollary~\ref{cor:A}. The \SAvP{} condition is stronger than the condition considered in \cite{Ursu:trace}. In fact, Ursu remarks in  \cite[Example 5.1]{Ursu:trace} that one can construct a (necessarily inner) action of $\Z/2 \times \Z/2$ on $M_2$ so that $M_2 \rtimes (\Z/2 \times \Z/2)$ is $M_4$. 

\begin{corollary} \label{cor:traces}
Let $\alpha \colon \Gamma \curvearrowright A$ be a \SAvP{} action of a discrete group $\Gamma$ on a unital \Cs{} $A$, and let $\tau$ be  an $\alpha$-invariant tracial state on $A$. Then $\bar{\tau} = \tau \circ E$ is the unique state on $A \rtimes_r \Gamma$ which extends $\tau$ and satisfies $\bar{\tau}(uxu^*)=\bar{\tau}(x)$, for all $x \in A \rtimes_r \Gamma$ and all unitaries $u \in A$. In particular, $\bar{\tau}$ is the unique  extension of $\tau$ to a tracial state on $A \rtimes_r \Gamma$.
\end{corollary}

\begin{proof} Any state $\rho$ on $A\rtimes_r \Gamma$ satisfying  $\rho(uxu^*)=\rho(x)$, for all $x \in A \rtimes_r \Gamma$ and $u \in \cU(A)$,  is constant on $C_A(x)$. Hence, by Lemma~\ref{lm:1} (iii), $\rho$ is determined by its restriction to $A$. 
\end{proof}

\noindent We have the following partial converse to the corollary: If $A$ is a unital separable monotracial \Cs, and $\alpha \colon \Gamma \curvearrowright A$ is an action such that 
the tracial state $\tau$ on $A$ has a unique extension to a state $\rho$ on $A \rtimes_r \Gamma$  satisfying $\rho(uxu^*)=\rho(x)$, for all $x \in A \rtimes_r \Gamma$ and all unitaries $u \in A$, then the extension of $\alpha$  to  $\pi_\tau(A)'' = A^{**}_\fin$ is (properly) outer. Hence the conclusion of Corollary~\ref{cor:traces} holds if and only if the action satisfies \SAvP, when $A$ is simple, separable and monotracial (possibly always, when $A$ is simple).

Indeed, we have an inclusion of finite von Neumann algebras $\pi_{\bar{\tau}}(A)'' \subseteq \pi_{\bar{\tau}}(A \rtimes_r \Gamma)''$, and $\pi_{\bar{\tau}}(A)'' \cong \pi_{\tau}(A)'' \cong A^{**}_\fin$, where $\bar{\tau}$ as above is the canonical extension of $\tau$ to the crossed product. If  the extension $\bar{\alpha}$ of $\alpha$ to $\pi_{\bar{\tau}}(A)''$ is not outer, then $\bar{\alpha}_t$ is inner for some $t \ne e$ in $\Gamma$, say $\bar{\alpha}_t = \mathrm{Ad}_w$. It follows that $wu_t^*$ belongs to the relative commutant $\pi_{\bar{\tau}}(A \rtimes_r \Gamma)'' \cap \pi_{\bar{\tau}}(A)'$, which therefore is non-trivial. Choose any non-trivial projection $q$ in this relative commutant, and consider the states $\rho_0(x) = \bar{\tau}(xq)/\bar{\tau}(q)$ and $\rho_1(x) = \bar{\tau}(x(1-q))/\bar{\tau}(1-q)$ on $\pi_{\bar{\tau}}(A \rtimes_r \Gamma)''$. Then $\rho_j \circ \pi_{\bar{\tau}}$ are two distinct states on $A \rtimes_r \Gamma$ that restrict to tracial states on $A$ (hence extend $\tau$), and they are both invariant under conjugation by unitaries from $A$.

\medskip \noindent
Recall that a unital \Cs{} $A$ is said to have the \emph{Dixmier property} if $C_A(x) \cap \C 1_A \ne \emptyset$\footnote{See Footnote 1.}, for all $x \in A$, and that an inclusion $B \subseteq A$ of unital \Cs s has the \emph{relative Dixmier property} if $C_B(x) \cap \C 1_A \ne \emptyset$, for all $x \in A$.  Popa introduced in \cite{Popa:ENS1999} and \cite{Popa:JFA2000} the notion of the relative Dixmier property, both in the context of von Neumann algebras and of $C^*$-algebras, and he proved several results characterizing which inclusions have the relative Dixmier property, including the result below (with different assumptions, but the proof is basically the same). 

Popa proved Lemma~\ref{lm:A} above assuming $A$ is simple with the Dixmier property and under suitable outerness conditions on the automorphisms $\alpha_j$ by showing that the inclusion of $A$ into $M_{n+1}(A)$, given by $a \mapsto \mathrm{diag}(a,\alpha_1(a), \dots, \alpha_n(a))$, has the relative Dixmier property. Popa used this to obtain the corollary below under the same outerness conditions on the action of $\Gamma$ on $A$ and assuming $A$ is simple. Note that we can skip the simplicity assumption. We refer to \cite{ArchRobTiku:Dixmier} for examples of non-simple \Cs s with the Dixmier property, and note here that the Cuntz-Toeplitz algebra mentioned in Example~\ref{ex:B} above is such an example of a non-simple \Cs{} with the Dixmier property. I thank Tron Omland for suggesting that the second part of the ``only if'' part of the corollary below holds.

\begin{corollary} \label{cor:RelDix} Let $\alpha \colon \Gamma \curvearrowright A$ be an action of a discrete group $\Gamma$ on a unital \Cs{} $A$. Then $A \subseteq A \rtimes_r \Gamma$ has the relative Dixmier property if and only if $A$ has the Dixmier property and the action $\alpha$ has \SAvP.
\end{corollary}

\begin{proof} It is clear that $A$ has the Dixmier property if $A \subseteq A \rtimes_r \Gamma$ has the relative Dixmier property. To see that this also implies that $\alpha$ has \SAvP, take $t \ne e$ in $\Gamma$ and $b \in A$. Since $C_A^{\, \alpha_t}(b) = C_A(bu_t)u_t^*$, cf.\ (the proof of) Proposition~\ref{prop:0},  it suffices to show that $0 \in C_A(bu_t)$.  We know that $C_A(bu_t) \cap \C 1_A \ne \emptyset$, because $A \subseteq A \rtimes_r \Gamma$ has the relative Dixmier property. We easily see that $C_A(bu_t) \subseteq A u_t$, which implies $C_A(bu_t) \cap \C 1_A \subseteq A u_t \cap \C 1_A = \{0\}$.

Suppose now  that the action has \SAvP. Let $x  \in A \rtimes_r \Gamma$. We show that if $\lambda \in \C$ and $\lambda \cdot 1_A \in C_A(E(x))$, then $\lambda \cdot 1_A  \in C_A(x)$. This will show that $A \subseteq A \rtimes_r \Gamma$ has the relative Dixmier property if $A$ has the Dixmier property.

Let $\ep >0$ and find unitaries $v_1, \dots, v_m \in A$ such that
$\big\|\frac{1}{m} \sum_{j=1}^m v_j E(x)v_j^* - \lambda \cdot 1_A\big\| < \ep$. Put $y = \frac{1}{m} \sum_{j=1}^m v_jx v_j^*$. Then $y \in C_A(x)$ and $\|E(y) - \lambda \cdot 1_A\| < \ep$. Put $z = y + (\lambda \cdot 1_A - E(y))$, so that $E(z) = \lambda \cdot 1_A$. Then $\lambda \cdot 1_A \in C_A(z)$, by Lemma~\ref{lm:1} (i), which implies that 
$$\mathrm{dist}(\lambda \cdot 1_A, C_A(x)) \le \mathrm{dist}(\lambda \cdot 1_A, C_A(y)) \le \|y-z\| < \ep,$$
thus proving the claim.
\end{proof}

\noindent
The theorem below is our main application of the strong averaging property. Recall from \cite{Ror:C*-irreducible} that an inclusion $A \subseteq B$ of unital \Cs s is \emph{$C^*$-irreducible} if all intermediate \Cs s are simple. It is a classic result that goes back to Elliott, \cite{Elliott:PRIM-1980}, Kishimoto, \cite{Kis:Auto} and Olesen-Pedersen, \cite{OlePed:C*-dynamicIII} that $C^*$-irreducibility of an inclusion $A \subseteq  A \rtimes_r \Gamma$ holds if and only if $A$ is simple and the action $\Gamma \curvearrowright A$ is outer, cf., \cite[Theorem 5.8]{Ror:C*-irreducible}.

Part (ii) about the Galois correspondance between intermediate \Cs s and intermediate groups was first obtained by Izumi, \cite{Izumi:Crelle2002}, in the case where $\Gamma$ is finite, the action $\Gamma \curvearrowright A$ is outer, $\Lambda$ is the trivial group, and $A$ is simple, and it was extended to arbitrary discrete groups $\Gamma$ by Cameron-Smith, \cite{CamSmith:Galois}, again assuming  $\Gamma \curvearrowright A$ is outer, $\Lambda$ is the trivial group, and $A$ is simple. This was further extended by Bedos and Omland, \cite[Theorem 5.3]{BedosOmland:irreducible}, to the general case of (ii) below, assuming that the (twisted) action of $\Gamma/\Lambda$ on $A$ is outer, in which case they can deduce their result from the  Cameron-Smith theorem. Our proof below is self-contained, and uses normality of $\Lambda$ in $\Gamma$ in a different way. Part (i) also follows from  \cite[Theorem 5.3]{BedosOmland:irreducible} in the case of $\Lambda$ being normal in $\Gamma$, but we exend their result to cover the general, not necessarily normal, case. 

A proof of (ii) below, in the case where $\Lambda$ is the trivial group and $A$ is simple, using Popa's averaging technique, was also given in \cite{Ror:C*-irreducible}. The statement below is more general than the one in \cite{Ror:C*-irreducible}, and  the proof has been futher streamlined.

\begin{theorem}  \label{thm:B}
Let $\Gamma \curvearrowright A$ be an \SAvP{} action of a discrete group $\Gamma$ on a unital \Cs{} $A$, and let $\Lambda$ be a subgroup of $\Gamma$ such that the action of $\Lambda \curvearrowright A$ is minimal. Then:
\begin{enumerate}
\item The inclusion $\cA \rtimes_r \Lambda \subseteq \cA \rtimes_r \Gamma$ is $C^*$-irreducible.
\item If, moreover $\Lambda$ is normal in $\Gamma$, then there is a bijective correspondance between intermediate \Cs s $\cA \rtimes_r \Lambda \subseteq \cD \subseteq \cA \rtimes_r \Gamma$ and intermediate subgroups $\Lambda \subseteq \Upsilon \subseteq \Gamma$ via $\cD = \cA \rtimes_r \Upsilon$. 
\end{enumerate}
 \end{theorem}
 
 \begin{proof}  For the proof of both parts of the theorem we use the following set-up: For each intermediate \Cs{} $A \rtimes_r \Lambda \subseteq D \subseteq A \rtimes_r \Gamma$, for each closed two-sided ideal $J$ of $D$, and for each $t \in \Gamma$, we consider the subset $E_t(J) =  \{ E_t(y) : y \in J\}$ of $A$.
 For each $y \in A \rtimes_r \Gamma$, each $a \in A$, and each $s,t \in \Gamma$ we have the identities:
 \begin{equation} \label{eq:1}
 a E_t(y) = E_t(ay), \qquad E_t(y) a= E_t( y\alpha_{t^{-1}}(a)),
 \end{equation}
 \begin{equation} \label{eq:2}
 \alpha_s(E_t(y)) = u_s E(yu_t^*)u_s^* = E(u_sy u_t^*u_s^*) = E_t(\alpha_s(y)u_{st^{-1}s^{-1}t}).
 \end{equation}
 It is clear that $E_t(J)$ is a closed linear subspace of $A$, and it follows from \eqref{eq:1} that it moreover is a  closed two-sided ideal in $A$. By \eqref{eq:2} and the fact that $D$ and $J$, are $\Lambda$-invariant, because $D$ contains $A \rtimes_r \Lambda$, it follows that $E_t(J)$ is $\Lambda$-invariant for $s \in \Lambda$ if $st^{-1}s^{-1}t \in \Lambda$. This holds in particular when $t = e$.
 
 (i). Let $A \rtimes_r \Lambda \subseteq D \subseteq A \rtimes_r \Gamma$ be an intermediate \Cs. We show that $D$ is simple. Take a non-zero closed two-sided ideal $J$ in $D$. By faithfulness of the conditional expectation $E$ and by the arguments above, $E(J)=E_e(J)$ is a non-zero closed two-sided $\Lambda$-invariant ideal of $A$, hence $E(J)=A$ by minimality of the action $\Lambda \curvearrowright A$.   Accordingly, we can find $y \in J$ with $E(y) = 1_A$. It therefore follows from Lemma~\ref{lm:1} (i) that $1_A \in C_A(y) \subseteq J$, which implies that $J=A$, and we are done.

 (ii). Suppose now that $\Lambda$ is normal in $\Gamma$. Let $A \rtimes_r \Lambda \subseteq D \subseteq A \rtimes_r \Gamma$ be an intermediate \Cs. Let $t \in \Gamma$. As shown above, $E_t(D)$ is a $\Lambda$-invariant closed two-sided ideal in $A$, because $\Lambda$ is normal in $\Gamma$, so either $E_t(D) = 0$ or $E_t(D) = A$ because $\Lambda \curvearrowright A$ is minimal. In the latter case, there exists $y \in D$ such that $E_t(y) = 1_A$. By the assumption that $\Gamma \curvearrowright A$ has \SAvP,  Lemma~\ref{lm:1}~(ii) yields that $u_t \in C_A^{\, \alpha_{t^{-1}}}(y)  \subseteq D$. In summary, we have shown that
 \begin{equation} \label{eq:3}
 E_t(D) \ne 0 \implies E_t(D) = A \implies u_t \in D.
 \end{equation}
 
 Now, let $\Upsilon = \{t \in \Gamma: u_t \in D\}$. It is clear that $\Upsilon$ is a subgroup of $\Gamma$ that contains $\Lambda$, and it is also clear that $A \rtimes_r \Upsilon  \subseteq D$. To prove the reverse implication, let $x \in D$. If $t \in \mathrm{supp}(x)$, then $u_t \in D$ by \eqref{eq:3}, which shows that $\mathrm{supp}(x) \subseteq \Upsilon$. This implies that $x \in A \rtimes \Upsilon$, as  can  be seen using the formula for the standard conditional expectation $A \rtimes_r \Gamma \to A \rtimes_r \Upsilon$, see the lemma below for details, or see  the  proof of  \cite[Theorem 3.5]{CamSmith:Galois}.\footnote{I thank Roger Smith for pointing out that this part of the argument is less obvious than it would seem at first sight.}
 \end{proof}

 \begin{lemma} Let $\alpha \colon \Gamma \curvearrowright A$ be an action of a discrete group $\Gamma$ on a \Cs{} $A$, let $\Upsilon$ be a subgroup of $\Gamma$, and let $x \in A \rtimes_r \Gamma$. Then $x \in A \rtimes \Upsilon$ if and only if $\mathrm{supp}(x) \subseteq \Upsilon$
  \end{lemma}
  
  \begin{proof} The ``only if'' part is clear. Assume that $\mathrm{supp}(x) \subseteq \Upsilon$. We may assume that $A \subseteq B(H)$, so that $A \rtimes_r \Gamma$ is represented on the Hilbert space $\tilde{H}=\ell^2(\Gamma,H) = \ell^2(\Gamma) \otimes H$ via the covariant representation $\pi \times \lambda$ given by 
  $$\pi(a)(\delta_g \otimes v) = \delta_g \otimes \alpha_{g^{-1}}(a)v, \quad \lambda(t)(\delta_g \otimes v) = \delta_{tg} \otimes v, \qquad g, t \in \Gamma, \; v \in H, \; a \in A.$$
 Write $\Gamma = \bigcup_{i \in I} \Upsilon t_i$ as a disjoint union of right cosets, and let $E_i$ be the projection from $\tilde{H}$ onto $\ell^2(\Upsilon t_i,H)$ (viewed as a subspace of $\tilde{H}$). Let $\Psi$ be the ucp map on $B(\tilde{H})$ given by $\Psi(T) = \sum_{i \in I} E_iTE_i$. We then have a commuting diagram:
  $$ \xymatrix{ A \rtimes_r \Gamma \ar[d]_{\mathcal{E}} \ar[r]^-{\pi \times \lambda} & B(\tilde{H}) \ar[d]^-\Psi \\
  A \rtimes_r \Upsilon  \ar[r]^-{\pi \times \lambda} & B(\tilde{H})}$$
 with $\mathcal{E}$ a conditional expectation.  We show that we have $\mathcal{E}(x) = x$, or, equivalently, that $E_i\, (\pi \times \lambda)(x)\, E_j = 0$, for $i \ne j$. Write $x$ as a formal sum  $x = \sum_{t \in \Gamma} a_t u_t$, and use that the sum $(\pi \times \lambda)(x)= \sum_{t \in \Gamma} \pi(a_t) \lambda(t)$ converges in $\tilde{H}$ on vectors of the form $\delta_g \otimes v$, to get
 $$
 E_i \, (\pi \times \lambda)(x)\, E_j (\delta_g \otimes v) = E_i \, \sum_{t \in \Gamma} \pi(a_t) \lambda(t) (\delta_g \otimes v) = \sum_{t \in \Gamma}  E_i (\delta_{tg} \otimes \alpha_{g^{-1}t^{-1}}(a_t)v)=0,
 $$
 for $g \in \Upsilon t_j$, because if $tg \in \Upsilon t_i$, then $t \notin \Upsilon$, since $i \ne j$, which entails $a_t=0$. 
  \end{proof}
 
 \begin{question} \label{q:A}
 Does (ii) of the theorem above hold also for non-normal subgroups $\Lambda$? 
 \end{question}
 
 \noindent  The question has an affirmative answer  if and only \eqref{eq:3} holds, for all intermediate sub-\Cs s $D$ and for all $t \in \Gamma$.  The second implication in \eqref{eq:3} holds just assuming the action $\Gamma \curvearrowright A$ has \SAvP. Normality of $\Lambda$, and minimality of the action $\Lambda \curvearrowright A$, are used to prove the first implication of \eqref{eq:3}.
 
A possible counterexample to Question~\ref{q:A} could be constructed by finding a dynamical system $\Gamma \curvearrowright A$, a (non-normal) subgroup $\Lambda$ of $\Gamma$, and a family of closed two-sided ideals $\{I_t\}_{t \in \Gamma}$ of $A$ such that $I_t = A$, for all $t \in \Lambda$, and $0 \ne I_t \ne A$, for at least one $t$, and such that
 $$D = \{y \in A \rtimes_r \Gamma: E_t(y) \in I_t, \; \text{for all} \; t \in \Gamma\},$$
 is \Cs. The latter holds if $I_t \, \alpha_t(I_s) \subseteq I_{ts}$ and $\alpha_t(I_{t^{-1}}) = I_{t}$, for all $s,t \in \Gamma$.
 
 Along the same lines, let us notice that the requirement in Theorem~\ref{thm:B} (ii) that the action of $\Lambda$ on $A$ is minimal is necessary. Indeed, if $J$ is a non-trivial $\alpha$-invariant ideal in $A$, consider the set $D$  consisting of all $x \in A \rtimes_r \Gamma$ with $E_t(x) \in J$, for all $t \in \Gamma \setminus \Lambda$. Then $D$ is seen to be an intermediate \Cs, which is clearly not equal to $A \rtimes_r \Upsilon$, for any intermediate group $\Upsilon$.

% \newpage
{\small{
\bibliographystyle{amsplain}
%\bibliography{operator} %undelete to use operator.bib file instead

%delete lines below to use operator.bib file
\providecommand{\bysame}{\leavevmode\hbox to3em{\hrulefill}\thinspace}
\providecommand{\MR}{\relax\ifhmode\unskip\space\fi MR }
\providecommand{\MRhref}[2]{%
  \href{http://www.ams.org/mathscinet-getitem?mr=#1}{#2}
}
\providecommand{\href}[2]{#2}

}}

\vspace{1cm}

\noindent
Mikael R\o rdam \\
Department of Mathematical Sciences \\
University of Copenhagen \\ 
Universitetsparken 5, DK-2100, Copenhagen \O \\
Denmark \\
Email: rordam@math.ku.dk\\
WWW: http://web.math.ku.dk/$\sim$rordam/

\end{document}